\documentclass[reqno]{amsart}
\usepackage{amssymb,latexsym,amsmath,amsthm,enumerate,amsbsy}
\usepackage[mathscr]{eucal}
\usepackage{framed,color,graphicx}
\usepackage{mathrsfs}
\usepackage[all]{xy}
\usepackage{tikz}
\usetikzlibrary{positioning,decorations.pathreplacing,patterns,decorations.pathmorphing}
\tikzset{%
element/.style={draw, shape=circle, fill=white, inner sep=1.4pt}
}

\DeclareSymbolFont{bbold}{U}{bbold}{m}{n}
\DeclareSymbolFontAlphabet{\mathbbold}{bbold}

\theoremstyle{plain}
\newtheorem{thm}{Theorem}[section]
\newtheorem{lem}[thm]{Lemma}
\newtheorem{cor}[thm]{Corollary}
\newtheorem{pro}[thm]{Proposition}
\newtheorem{example}[thm]{Example}

\theoremstyle{definition}

\newcommand{\up}[1]{\textup{#1}}

\newcommand{\bigand}{\operatornamewithlimits{\hbox{\Large$\&$}}}

\newcommand{\bell}{\boldsymbol{\ell}}

\newcommand{\bk}{\mathbf{k}}

\newcommand{\bp}{\mathbf{p}}

\newcommand{\bs}{\mathbf{s}}

\newcommand{\bv}{\mathbf{v}}
\newcommand{\bw}{\mathbf{w}}

\begin{document}

\title[Limit varieties]{Flat extensions of groups and limit varieties of ai-semirings}

\author{Miaomiao Ren}
\address{School of Mathematics, Northwest University, Xi'an, 710127, Shaanxi, P.R. China}
\email{miaomiaoren@nwu.edu.cn}
\author{Marcel Jackson}
\address{Department of Mathematics and Statistics\\ La Trobe University\\ Victoria  3086\\
Australia} \email{M.G.Jackson@latrobe.edu.au}
\author{Xianzhong Zhao}
\address{School of Mathematics, Northwest University, Xi'an, 710127, Shaanxi, P.R. China}
\email{zhaoxz@nwu.edu.cn}
\author{Donglin Lei}
\address{School of Mathematics, Northwest University, Xi'an, 710127, Shaanxi, P.R. China}
\email{dllei@nwu.edu.cn}


\thanks{Miaomiao Ren is supported by the National Natural Science Foundation of China (11701449)
and the Natural Science Foundation of Shaanxi Province (2022JM-009). Xianzhong Zhao is supported by the National Natural Science Foundation of China (11971383, 11571278).}
\begin{abstract}
The present paper is a continuation of \cite{jrz} and is devoted to the study of limit varieties
of additively idempotent semirings.
A limit variety is a nonfinitely based variety whose proper subvarieties are all finitely based.
We present concrete constructions for one infinite family of limit additively idempotent semiring varieties, and one further ad hoc example.  Each of these examples can be generated by a finite flat semiring, with the infinite family arising by a way of a complete characterisation of limit varieties that can be generated by the flat extension of a finite group.  We also demonstrate the existence of other examples of limit varieties of additively idempotent semirings, including one further continuum-sized family, each with no finite generator, and two further ad hoc examples.  While an explicit description of these latter examples is not given, one of the examples is proved to contain only trivial flat semirings.
\end{abstract}

\maketitle
\section{Introduction and preliminaries}
A variety is called \emph{finitely based} if it can be defined by finitely many identities.
Otherwise, it is called \emph{nonfinitely based}.  The finite basis problem for
various varieties has been intensively studied since Lyndon's work \cite{lyn51,lyn56} in the 1950s; see the surveys such as \cite{mc02, vol} for example.

A variety is said to be \emph{hereditarily finitely based} if all of its subvarieties are finitely based.
A variety is called \emph{limit} if it is nonfinitely based, but each of its proper
subvarieties is finitely based. In other words, a variety is limit if and only if
it is a minimal nonfinitely based variety.  By Zorn's Lemma, every nonfinitely based variety
contains a limit variety. It follows that a variety is hereditarily finitely based if and only if
it contains no limit subvarieties.  Therefore, classifying hereditarily finitely based varieties in a certain sense reduces
to classifying limit varieties.  As observed by
Lee and Volkov \cite{lv}, there are a number of challenges to classifying limit varieties.  In general, the problem of determining if a finite algebra is hereditarily finitely based is undecidable \cite[Theorem~7.1]{jac05}, and even describing a single example of a limit variety can be difficult.  Despite continuum many limit varieties of groups known to exist, no explicit description of even a single one is known; see Kozhevnikov \cite{koz, koz2}.  Within semigroups, and more recently monoids, there has been some progress
on finding explicit examples; see \cite{gu20, glz21, gs21,  jac05, lv, sap, ol21,zl16}.  We mention in particular \cite{gs21}, the culmination of several of these references, where a complete classification is achieved for finitely generated $\mathscr{J}$-trivial monoid limit varieties: there are precisely~7 of them. However, to the best of our knowledge,
examples of limit varieties of ai-semirings have not been known.
The aim of this paper is to present a series of examples of such varieties.  In Theorem \ref{thm1} we present a concrete construction for a countably infinite family of finitely generated limit varieties and in Theorem \ref{thm2} we provide an explicit description of a combinatorial example (that is, finitely generated and containing only trivial subgroups for its multiplicative reduct).  We also demonstrate the existence of a continuum of limit varieties that cannot be generated by any finite example (Theorem \ref{thm:infinitegroup}) as well as the existence of two further ad hoc examples (Theorem \ref{thm:maxplus} and Theorem \ref{thm:anotherlimit}), one of which we show cannot be generated by a flat semiring.

A \emph{semiring} $(S, +, \cdot)$ is an algebra with two binary operations $+$ and $\cdot$ such that the additive reduct $(S, +)$ is a commutative semigroup,
the multiplicative reduct $(S, \cdot)$ is a semigroup and $S$ satisfies the distributive laws
\[
(x+y)z\approx xy+xz,\quad x(y+z)\approx xy+xz.
\]
The set of all natural numbers under usual addition and multiplication is a natural
example of a semiring, but one can also easily find many other examples of semirings
in almost all branches of mathematics.
Semirings can be regarded as a common generalization of both rings and distributive lattices.
A semiring is called an \emph{additively idempotent semiring} (ai-semiring for short) if its additive reduct is a semilattice, that is,
a commutative idempotent semigroup. This family of semirings includes the Kleene semiring of regular languages (see Conway \cite{con}),
the max-plus algebra (see Aceto, \'Esik and Ing\'olfsd\'ottir \cite{aei}),  the semiring of all binary relations on a set
(see Andr\'eka and Mikulas \cite{andmik}) and the matrix semiring over an ai-semiring (see Baeth and Gotti \cite{bg}).
These and other similar algebras have broad applications in tropical geometry, information science and theoretical computer science
(see \cite{gl, go, ms}).

A number of papers in the literature have considered various ai-semiring varieties, for example,
see \cite{aei, dol07, gpz05, jrz, pas05, rzs4, rzw, sr, vol21, zrcsd} amongst many others.
In particular, Dolinka \cite{dol07} found the first example of a finite nonfinitely based ai-semiring.
Pastijn et al. \cite{gpz05, pas05} showed that there are exactly 78 ai-semiring varieties satisfying
the identity $x^2\approx x$, which are all finitely based.
Ren et al. \cite{rzw} answered the finite basis problem for ai-semiring varieties satisfying
the identity $x^3\approx x$ and proved that there are 179 such varieties.
Gusev and Volkov \cite{gv22a, gv22b} considered the equational theory of some ai-semirings,
which include the power semirings of finite groups and the
finite ai-semirings whose multiplicative reducts are inverse semigroups.
Volkov \cite{vol21} proved that the ai-semiring $B_2^1$ whose multiplicative reduct is the 6-element Brandt monoid has no finite basis
for its identities. Zhao et al. \cite{zrcsd} detailly investigated ai-semirings of order 3
and showed that all but possible one (which is denoted by $S_7$) are finitely based. Recently,
Jackson et al. \cite{jrz} presented some general results on the finite basis problem for ai-semiring varieties.
As applications, it is proved that many finite ai-semirings with finitely based semigroup reducts are nonfinitely based,
which include the semirings $B_2^1$ and $S_7$.
The present paper is a further development and refinement of \cite{jrz}.
We shall systematically study limit varieties of ai-semirings.
For this prupose, the following notions and notations are necessary.

By a \emph{flat semiring} we mean an ai-semiring such that its multiplicative reduct has a zero element $0$
and $a+b=0$ for all distinct elements $a$ and $b$.
Jackson et al.~\cite[Lemma 2.2]{jrz} observe that a semigroup with $0$ becomes a flat semiring with the top element $0$ if and only if it is
$0$-cancellative: $ab=ac\neq0$ implies $b=c$ and $ba=ca\neq0$ implies $b=c$
for all $a, b, c\in S$.
If $S$ is a cancellative semigroup, then $S^0$ is $0$-cancellative
and becomes a flat semiring. It is called the {\it flat extension}\footnote{We mention that this notation is also used for flat extensions of other objects, but this article concerns only the case when the flat extension forms a semiring.} of $S$
and is denoted by $\flat(S)$.
The flat extensions of groups are particularly interesting and play an important role in this paper.
%

Let $W$ be a nonempty subset of the free commutative semigroup over $X$,
and let $S_c(W)$ denote the set of all nonempty subwords of words in
$W$ together with a new symbol $0$. Define a binary operation $\cdot$ on $S_c(W)$ by the rule
\begin{equation*}
u\cdot v=
\begin{cases}
uv& \text{if }~uv\in S_c(W)\setminus \{0\}, \\
0& \text{otherwise.}
\end{cases}
\end{equation*}
Then $(S_c(W), \cdot)$ forms a semigroup with a zero element $0$.
It is easy to see that $(S_c(W), \cdot)$ is $0$-cancellative and so $S_c(W)$ becomes a flat semiring.
In particular, if $W$ consists of a single word $w$ we shall write $S_c(W)$ as $S_c(w)$.  If we allow the empty word in this construction, then the semigroup reduct is a monoid, and we use the notation $M_c(W)$.  Finally, we also mention the non-commutative variants of these (using the free semigroup and free monoid over $X$), which are simply denoted by $S(W)$ and $M(W)$.
Correspondingly, we have the notations $S(w)$ and $M(w)$.

For other notations and terminology used in this paper, the reader
is referred to Jackson et al.~\cite{jrz} for background on flat semirings, and to Burris and Sankappanavar \cite{bur} for
information concerning universal algebra. We shall assume that the
reader is familiar with the basic results in these areas.

\section{Flat extensions of groups}\label{sec:groups}
In this section we provide an infinite series of limit ai-semiring varieties from the approach of
flat extensions of groups.

For a detailed treatment of the following concepts one may consult a text such as Gorbunov \cite{gor}, or reference such as \cite{jac08} with similar context to the present article.  Recall that the \emph{quasivariety} generated by a class of similar algebras $K$ is $\mathsf{Q}(K)=\mathsf{SPP}_{\rm u}(K)$, where $\mathsf{S}$, $\mathsf{P}$ and $\mathsf{P}_{\rm u}$ denote closure under isomorphic copies of substructures, direct products and ultraproducts, respectively.   Equivalently, it is the class of algebras of the same signature as $K$ that satisfy the quasiequations holding in $K$: universally quantified sentences of the form $\big(\bigand_{1\leq i\leq n}\alpha_i\big)\Rightarrow \alpha_0$ where $\alpha_0,\dots,\alpha_n$ are equalities and $n\geq 0$.  As with varieties, we may speak of finitely based (or finitely axiomatised) quasivarieties, and of limit quasivarieties.  We apply the phrase \emph{finitely q-based} to algebras or classes whose quasiequations admit a finite basis.  We note also that when  $K$ is a finite class of finite algebras then $\mathsf{Q}(K)=\mathsf{SP}(K)$.  The notation $\mathsf{V}(K)$ will denote the variety generated by $K$.

Let us extend $\flat$ extend to classes $K$ by $\flat(K)=\{\flat(G)\mid G\in K\}$ and let $B(d)$ denote the variety (hence quasivariety) of groups whose exponent divides the number~$d$; that is, satisfying $x^dy\approx yx^d\approx y$.   The following result is a restriction of Jackson \cite[Corollary~5.4, Theorem~5.12]{jac08}.\footnote{Corollary 5.4 and Theorem 5.12 of \cite{jac08} are expressed in terms of what is called \emph{pointed semidiscriminator extensions} of partial algebras, producing \emph{pointed semidiscriminator varieties} and refers to universal Horn classes rather than quasivarieties.  As explored in \cite[\S7]{jac08} however, flat extensions of groups of finite exponent are term equivalent to the corresponding pointed semidiscriminator extensions, and nonempty universal Horn classes are quasivarieties because all groups contain the trivial subgroup.  If the empty universal Horn class, definable by $\forall x\ x\not\approx x$, is also allowed, then the isomorphism in Theorem \ref{thm:correspondence} extends to include the trivial variety in its range.}
\begin{thm}\label{thm:correspondence}
\begin{enumerate}
\item
The operator $K\mapsto \mathsf{V}(\flat(K))$ is a lattice isomorphism between the lattice of subquasivarieties of $B(d)$ and the lattice of nontrivial subvarieties of $\mathsf{V}(\flat(B(d)))$.
\item
The operator $K\mapsto \mathsf{V}(\flat(K))$ preserves local finiteness, finite axiomatisability and nonfinite axiomatisability, and for any class $K\subseteq B(d)$, the class of subdirectly irreducible members of $\mathsf{V}(\flat(K))$ is the isomorphism closure of $\{\flat(G)\mid G\in \mathsf{Q}(K)\}$.
\end{enumerate}
\end{thm}
The operator in Theorem \ref{thm:correspondence} is a lattice isomorphism preserving both the finite basis and nonfinite basis properties, so provides a bijective correspondence between limit quasivarieties and limit varieties.  We collect these consequences of \cite{jac08} in the following lemma.
\begin{lem}\label{lem:group}
Let $G$ be a group of finite exponent.
\begin{enumerate}
\item
$\flat(G)$ is finitely based \up(for its equational theory as a semiring\up) if and only if $G$ has a finite basis for its quasi-equational theory \up(as a group\up).
\item $\flat(G)$ generates a limit variety of semirings if and only if $G$ generates a limit quasivariety of groups.
\end{enumerate}
\end{lem}
Ol$'$shanski\u{\i}~\cite{ols} showed that a finite group is finitely q-based if and only its Sylow subgroups are abelian.  Lemma \ref{lem:group}(1) can then be refined in the form of the following lemma.
\begin{lem}\label{lem1}
Let $G$ be a finite group.
Then $\flat(G)$ is finitely based if and only if all Sylow subgroups of $G$ are abelian.
\end{lem}
A consequence of Ol$'$shanski\u{\i}'s classification is the following lemma.
\begin{lem}\label{lem:2gen}
Let $G$ be a finite group with a nonabelian Sylow subgroup.  Then there is a finite group $H$ that is a $2$-generated $p$-group that generates a limit subquasivariety of the quasivariety of $G$ and moreover every limit subquasivariety of $\mathsf{Q}(G)$ is of this form.

Equivalently, let $\flat(G)$ be the flat extension of a finite group $G$ with a nonabelian Sylow subgroup. Then there is a finite group $H$ that is a $2$-generated $p$-group with $\flat(H)$ generating a limit subvariety of the variety of $\flat(G)$, and all limit subvarieties of $\mathsf{V}(\flat(G))$ are of this form.
\end{lem}
\begin{proof}
We prove the first statement only because equivalence with the second statement is immediate from Theorem \ref{thm:correspondence} and Lemma~\ref{lem:group}.

Let $S$ be a nonabelian Sylow subgroup of $G$.  Then $S$ is a $p$-group for some prime~$p$ and generates a subquasivariety of that of $G$.  As $S$ is a Sylow subgroup of itself, and is finite, it follows from \cite{ols} that $S$ is not finitely q-based.  As $S$ is a $p$-group it is nilpotent and satisfies $x^{p^n}\approx 1$ for some $n$.  Let $L$ be a limit subquasivariety of $\mathsf{Q}(S)$, which evidently is also a locally finite quasivariety of nilpotent groups satisfying $x^{p^n}\approx 1$.  Let $H$ be any member of $L$ that is generated by two non-commuting elements, which exists as $L=\mathsf{SPP}_{\rm u}(L)$ is closed under taking subgroups and contains nonabelian members (abelian quasivarieties are varieties and are determined by the maximal exponent).  As $L$ is locally finite, $H$ is finite.  As $H$ is nilpotent and nonabelian it is nonfinitely based by \cite{ols}; it is a $p$-group as it satisfies $x^{p^n}\approx 1$.  As $L$ is a limit quasivariety containing the nonfinitely q-based quasivariety generated by $H$, it follows that $H$ generates $L$, as required.
\end{proof}
One can also easily deduce that $H$ is nilpotent of class 2, though this is not used in the remainder of the proof.
%

A group is called \emph{minimal nonabelian} if it is
nonabelian, but each of its proper subgroups is abelian.  Lemma \ref{lem:2gen} shows that all limit quasivarieties within the quasivariety generated by a finite group (with nonabelian Sylow subgroup) are generated by minimal nonabelian $p$-groups. Our goal is to find which of these generate quasivarieties that are minimal with respect to not being finitely q-based.

\begin{lem}\label{lem: group1}
Let $G$ be a finite minimal nonabelian group. If $H$ is a finite nonabelian group in $\mathsf{Q}(G)$,
then $|G|\leq |H|$.
\end{lem}
\begin{proof}
Suppose that $H$ is a finite nonabelian group in $\mathsf{Q}(G)$. Then there exists $n\geq 1$ such that
$H$ is a nonabelian subgroup of the direct product $\prod_{1\leq i\leq n} G_i$ of $n$ copies $G_i$ of $G$, $1\leq i\leq n$.
For $1\leq i\leq n$, let $\pi_i: \prod_{1\leq j\leq n} G_j \to G_i$ denote the projection map on the $i$th coordinate of $\prod_{j\in I} G_j$.
Then $\pi_i(H)$ is a subgroup of $G$ for all $1\leq i\leq n$. Since $H$ is nonabelian, it follows that
there exists $1\leq i\leq n$ such that $\pi_i(H)$ is a nonabelian subgroup of $G$. Furthermore,
we have that $\pi_i(H)=G$, since $G$ is minimal nonabelian. Thus $|G|\leq |H|$ as required.
\end{proof}

By Theorem \ref{thm:correspondence}(1) and Lemma \ref{lem: group1} we have the following corollary.
\begin{cor}\label{cor:quasi3}
Let $G_1$ and $G_2$ be finite minimal nonabelian groups. Then $\mathsf{Q}(G_1)=\mathsf{Q}(G_2)$ \up(equivalently, $\mathsf{V}(\flat(G_1))=\mathsf{V}(\flat(G_2))$\up)
if and only if $G_1$ and $G_2$ are isomorphic.
\end{cor}

\begin{pro}\label{pro: limit1}
Let $G$ be a finite nonabelian $p$-group. Then $\mathsf{V}(\flat(G))$ is a limit variety
if and only if $G$ is the unique minimal nonabelian group in $\mathsf{Q}(G)$ up to isomorphism.
\end{pro}
\begin{proof}
Suppose that $G$ generates a limit quasivariety. Let $H$ be an arbitrary finite minimal nonabelian group in $\mathsf{Q}(G)$.
As $H$ is not finitely q-based and $\mathsf{Q}(G)$ is a limit quasivariety it follows that $\mathsf{Q}(G)=\mathsf{Q}(H)$ and then Corollary \ref{cor:quasi3} implies that $G\cong H$.

Conversely, assume that $G$ is the unique minimal nonabelian group in $\mathsf{Q}(G)$.  As every limit subquasivariety of $\mathsf{Q}(G)$ (by Lemma \ref{lem:2gen}) is generated by a minimal nonabelian group, it follows that $\mathsf{Q}(G)$ is a limit quasivariety.
\end{proof}

Recall that a group $G$ is said to be \emph{metacyclic} if it contains a cyclic normal subgroup $N$ such
that $G/N$ is cyclic. It is easy to see that each divisor of a metacyclic group is also metacyclic.
The following result, due to R\'{e}dei \cite{re}, provides the classification of all finite minimal nonabelian $p$-groups.
\begin{lem}\label{lem: red}
Let $G$ be a finite minimal nonabelian $p$-group. Then $G$ is isomorphic to one of the following groups:
\begin{itemize}
\item [$(i)$] $Q_8:=\langle a, b \mid a^4=1, a^2=b^2, aba=b\rangle$, where $|Q_8|=8$;
\item [$(ii)$] $M_p(m, n):=\langle a, b \mid a^{p^m}=b^{p^n}=1, ab=ba^{1+p^{m-1}}\rangle$, $m\geq 2$, $n\geq 1$.  The group
              $M_p(m, n)$ is a metacyclic group of order $p^{m+n}$;
\item [$(iii)$] $M_p(m, n, 1):=\langle a, b \mid a^{p^m}=b^{p^n}=c^p=1, ab=bac, ac=ca, bc=cb\rangle$, $m\geq 1$, $n\geq 1$.  The group  $M_p(m, n, 1)$ is a nonmetacyclic group of order $p^{m+n+1}$ except $M_2(1, 1, 1)$.
\end{itemize}

In the above group presentations, the groups given by different parameters are not isomorphic except
$M_2(2, 1)$ and $M_2(1, 1, 1)$, which are both isomorphic to the Dihedral group of order $8$.
\end{lem}

\begin{pro}\label{pro: limit2}
Let $G$ be one of the following groups:
\begin{itemize}
\item [$(i)$] $Q_8$;
\item [$(ii)$] $M_p(m, n)$, $m\geq 2$, $n\geq 1$;
\item [$(iii)$] $M_p(m, n, 1)$, $m\geq 1$, $n\geq 1$ and $m\neq n$.
\end{itemize}
Then there exists a finite minimal nonabelian group $H$ such that $\mathsf{Q}(H)$ is a proper subquasivariety of $\mathsf{Q}(G)$.
\end{pro}
\begin{proof}
By Lemma \ref{lem: group1}
it is enough to show that there exists a finite minimal nonabelian subgroup $H$ of $G \times G$ such that $|H|>|G|$. Consider the following three cases:

\textbf{Case 1.} $G=Q_8$. Let $H$ denote the subgroup $H$ of $G\times G$ generated by $(a, 1)$ and $(ab, b)$.
Then it is easy to verify that $(a, 1)^4=(ab, b)^4=(1, 1)$ and that
$(a, 1)(ab, b)=(ab, b)(a, 1)^3=(b^3, b)$. By Lemma \ref{lem: red} $(ii)$ there exists an epimorphism
$\varphi:M_2(2, 2) \to H$. It is not difficult to show that $\varphi$ is injective and so it is an
isomorphism. Thus $H$ is isomorphic to $M_2(2, 2)$ and so it is
a minimal nonabelian $2$-group of order $16$. We conclude that $|H|>|G|$.

\textbf{Case 2.} $G=M_p(m, n)$, where $m\geq 2$, $n\geq 1$. Consider the following two subcases:

{\bf Subcase 2.1.} $n\geq m$. Let $H$ denote the subgroup $H$ of $G\times G$ generated by $(a, b)$ and $(b, 1)$.
Then $(a, b)^{p^n}=(b, 1)^{p^n}=(1, 1)$. Since $a^p$ is central in $G$, it follows that
$[(a, b), (b, 1)]=(a^{p^{m-1}}, 1)$ is central in $G\times G$ and that $[(a, b), (b, 1)]^p=(1, 1)$.
By Lemma \ref{lem: red} $(iii)$ there exists an epimorphism $\varphi: M_p(n, n, 1) \to H$.
Notice that $n\geq m$. We can show that $\varphi$ is an injective and so it is an isomorphism.
Thus $H$ is isomorphic to $M_p(n, n, 1)$ and so it is a minimal nonabelian $p$-group of order $p^{2n+1}$.
We therefore have
\[|H|=p^{2n+1}>p^{m+n}=|G|.\]

{\bf Subcase 2.2.} $m> n$. Consider the subgroup $H$ of $G\times G$ generated by $(a, b)$ and $(1, a)$.
Then $(a, b)^{p^m}=(1, a)^{p^m}=(1, 1)$ and $(1, a)(a, b)=(a, b)(1, a)^{p^{m-1}+1}$.
By Lemma \ref{lem: red} $(ii)$ there exists an epimorphism $\varphi: M_p(m, m) \to H$.
Furthermore, $\varphi$ is injective and so it is an isomorphism.
Thus $H$ is isomorphic to $M_p(m, m)$ and so it is a minimal nonabelian $p$-group of order $p^{2m}$.
Hence
\[|H|=p^{2m}>p^{m+n}=|G|.\]

\textbf{Case 3.} $G=M_p(m, n, 1)$, where $m\geq 1$, $n\geq 1$ and $m\neq n$. Assume that $m>n$.
Consider the subgroup $H$ of $G\times G$ generated by $(a, b)$ and $(1, a)$.
Then $(a, b)^{p^m}=(1, a)^{p^m}=(1, 1)$, $[(a, b), (1, a)]=(1, c^{-1})$ is central in $G\times G$ and $(1, c^{-1})^p=(1, 1)$.
By Lemma \ref{lem: red} $(iii)$ there exists an epimorphism $\varphi: M_p(m, m, 1)\to H$.
Furthermore, $\varphi$ is injective and so it is an isomorphism.
This shows that $H$ is isomorphic to $M_p(m, m, 1)$
and so it is a minimal nonabelian $p$-group of order $p^{2m+1}$.
Thus
\[|H|=p^{2m+1}>p^{m+n+1}=|G|.\]
This completes the proof.
\end{proof}

The following lemma can be found in \cite[Exercise 18]{ber08}.
\begin{lem}\label{ber1}
Let $G$ be a nilpotent group of class $2$, $x, y, z\in G$, $n\geq 1$. Then
\[
(xy)^n=x^ny^n[y, x]^{\binom{n}{2}}, \quad [x, y]^n=[x^n, y]=[x, y^n].
\]
\end{lem}

\begin{pro}\label{pro: limit3}
Let $G$ denote the group $M_p(m, n, 1)$, where $m=n\geq 2$ if $p=2$ and $m=n\geq 1$ if $p>2$.
If $H$ is a finite minimal nonabelian group in $\mathsf{Q}(G)$, then $\mathsf{Q}(H)=\mathsf{Q}(G)$.
\end{pro}
\begin{proof}
We first show that ${\rm exp}(G)=p^m$, where ${\rm exp}(G)$ denotes the least common multiple of the orders of all elements of $G$.
Let $\alpha$ be an arbitrary element of $G$. Then there exist $s_1, s_2, s_3\geq 1$ such that
$\alpha=a^{s_1}b^{s_2}c^{s_3}$. Furthermore, we have
\begin{align*}
{\alpha}^{p^m}
&=(a^{s_1}b^{s_2}c^{s_3})^{p^m}=(a^{s_1}b^{s_2})^{p^m}c^{s_3p^m}=(a^{s_1}b^{s_2})^{p^m}&\\
&=a^{s_1p^m}b^{s_2p^m}[b^{s_2}, a^{s_1}]^{\binom{p^m}{2}}\quad (\text{by Lemma}~\ref{ber1})\\
&=[b^{s_2}, a^{s_1}]^{\binom{p^m}{2}}\\
&=[b, a]^{s_1s_2\binom{p^m}{2}}  \quad (\text{by Lemma}~\ref{ber1})\\
&=c^{s_1s_2(p-1)\binom{p^m}{2}}&\\
&=1.&
\end{align*}
Thus ${\rm exp}(G)=p^m$.

Suppose now that $H$ is a finite minimal nonabelian group in $\mathsf{Q}(G)$. Then
$H$ is isomorphic to a subgroup of the direct product of finite copies of $G$.
Without loss of generality, we may assume that $H$ is a subgroup of the group $P$ of two copies of $G$.
Let $P_1$ denote the subgroup of $P$ generated by $(a, 1)$ and $(b, 1)$, and let $P_2$ denote the subgroup of $P$ generated by $(1, a)$ and $(1, b)$.
Then both $P_1$ and $P_2$ are normal and are isomorphic to $G$. This implies that $P$ is isomorphic to $P_1\times P_2$.
Since $H$ is nonabelian, it follows from Lemma \ref{lem: group1} that either $\pi_1(H)=G$ or $\pi_2(H)=G$.
Furthermore, we have that either $P_1H=P$ or $P_2H=P$. Assume that $P_1H=P$. Then
\[
G\cong P_2\cong P/{P_1}=P_1H/{P_1}\cong H/{(H\cap P_1)}.
\]
This shows that $H/{H\cap P_1}$ is nonmetacyclic and so is $H$.

By Lemma \ref{lem: red} there exist $s, t\geq 1$ such that $H$ is isomorphic to $M_p(s, t, 1)$.
Since ${\rm exp}(G)=p^m$, it follows that ${\rm exp}(H)\leq p^m$ and so $s, t\leq m$.
On the other hand, we have by Lemma \ref{lem: group1} that $|G|\leq |H|$ and so
$2m\leq s+t$. Hence $s=t=m$ and so $H$ is isomorphic to $G$. We conclude that $\mathsf{Q}(H)=\mathsf{Q}(G)$ as required.
\end{proof}

By Propositions \ref{pro: limit1}, \ref{pro: limit2} and \ref{pro: limit3} and Lemma \ref{lem: red}
we can establish the following theorem, which is the main result of this section.
\begin{thm}\label{thm1}
Let $G$ be a finite minimal nonabelian $p$-group. Then $\mathsf{V}(\flat(G))$ is a limit variety
if and only if $G$ is isomorphic to $M_p(m, n, 1)$ for some $m=n\geq 2$ if $p=2$ and some $m=n\geq 1$ if $p>2$.
\end{thm}

\begin{cor}
There is an infinite number of limit varieties of ai-semirings.
\end{cor}
\begin{proof}
This follows from Corollary \ref{cor:quasi3}, Lemma \ref{lem: red} and Theorem \ref{thm1}.
\end{proof}

\begin{cor}
The $27$-element group $M_3(1, 1, 1)$ is
the smallest group whose flat extension generates a limit variety.
\end{cor}
\begin{proof}
This is a consequence of Lemma \ref{lem: red} and Theorem \ref{thm1}.
\end{proof}

It is possible for limit varieties to have infinitely many subvarieties and to be not finitely generated.  As a consequence of Lemma \ref{lem:2gen}, it follows that when $G$ is a finite group and $\mathsf{V}(\flat(G))$ is nonfinitely based, then all limit subvarieties are finitely generated.  The next proposition (ultimately just an application of Theorem \ref{thm:correspondence} to the classification of \cite{ols}) shows that the subvariety lattice is always finite.  We use the notation $\oplus$ to denote the linear sum of lattices \cite{dav}.
\begin{pro}
Let $G$ be a finite group of exponent $d$ and let $L_d$ denote the lattice of divisors of $d$.  If $\mathsf{V}(\flat(G))$ generates a limit variety then the lattice of subvarieties of $\mathsf{V}(\flat(G))$ is isomorphic to $1 \oplus L_d \oplus 1$; in particular, $\mathsf{V}(\flat(G))$ has only finitely many subvarieties.
\end{pro}
\begin{proof}
Let $L$ denote the lattice of subvarieties of $\mathsf{V}(\flat(G))$ and $L_Q$ denote the lattice of subquasivarieties of $\mathsf{Q}(G)$.  By Theorem \ref{thm:correspondence}, the lattice $L$ is isomorphic to $1\oplus L_Q$; our goal is to show that $L_Q\cong L_d\oplus 1$.  All proper subquasivarieties of $\mathsf{Q}(G)$ are abelian, so that $\mathsf{Q}(G)$ covers the lattice of its abelian subquasivarieties.  The variety of abelian groups of exponent $d$ is generated by the cyclic groups whose orders divide $d$, and as the exponent of $G$ is the lowest common multiple of the orders of its elements, it follows that each such abelian group is a subgroup of $G$.  Hence the lattice of  abelian subquasivarieties of $\mathsf{Q}(G)$ is simply the lattice of varieties of abelian groups of exponent dividing $d$, which is isomorphic to $L_d$.
\end{proof}
There are also limit varieties of the form $\mathsf{V}(\flat(G))$ that are not generated by finite groups.  Indeed, \cite[Theorem 5 and 6]{koz2} implies that for every sufficiently large prime~$p$ there are groups $G$ of  exponent $p$ such that $\mathsf{V}(G)$ is a limit variety of groups, and whose proper subvarieties are abelian.  The variety $\mathsf{V}(G)$ is a quasivariety, so must also contain a limit subquasivariety.  Let $H$ be a group generating such a quasivariety; as before we may assume that $H$ is $2$-generated.   Now $H$ is nonabelian, and as all proper subvarieties of $\mathsf{V}(G)$ are abelian, it follows that $\mathsf{V}(H)=\mathsf{V}(G)$ so that $H$ is not locally finite; in particular $H$ is infinite, so is not of the form covered by Theorem \ref{thm1}.  There are continuum many varieties of this form.
\begin{thm}\label{thm:infinitegroup}
There are continuum many distinct limit varieties of the form $\mathsf{V}(\flat(G))$ that are not locally finite.
\end{thm}
Every quasivariety contains the free algebras for the variety it generates, so the smallest quasivariety generating a given variety is the quasivariety generated by the relatively free algebras.  So the generator $H$ for a limit quasivariety arising from Kozhevnikov's arguments in \cite{koz2} can be the 2-generated relatively free group in the corresponding limit variety.

\section{The variety of $S_c(abc)$}
In this section we present a further limit ai-semiring variety that contains no flat extensions of groups.
More precisely, we shall show that the variety $\mathsf{V}(S_c(abc))$ generated by $S_c(abc)$ is limit.
Let $\textbf{F}$ denote the variety generated by all flat semirings.
\begin{lem}\label{lem: flat}
\begin{itemize}
\item [$(i)$] $\textbf{F}$ is finitely based and each subdirectly irreducible member of this variety
              is a flat semiring that has a unique $0$-minimal multiplicative ideal;
\item [$(ii)$] $\mathsf{V}(S_c(abc))$ is nonfinitely based;
\item [$(iii)$] $\mathsf{V}(S_c(a))$ is finitely based, is a minimal nontrivial subvariety of $\mathsf{V}(S_c(abc))$
                and is determined within the variety $\mathsf{V}(S_c(abc))$ by the identity
\begin{align}
x_1x_2 & \approx y_1y_2. \label{i4}
\end{align}
\end{itemize}
\end{lem}

\begin{proof}
$(i)$ This follows from \cite[Lemma 2.1]{jrz} and the fact that there is a one-to-one order-preserving correspondence
between semiring congruences and multiplicative ideals on a flat semiring.

$(ii)$ This is a corollary of \cite[Theorem 4.9]{jrz}.

$(iii)$ First, it is trivial that $\mathsf{V}(S_c(a))$ is a subvariety of subvariety of $\mathsf{V}(S_c(abc))$ as $S_c(a)$ is isomorphic to a quotient (and a subsemiring) of $S_c(abc)$.   The remaining properties of $\mathsf{V}(S_c(a))$  can be found in \cite{sr}.
\end{proof}
In the proof of the following proposition, we will refer to \emph{prime} elements of a semiring.  These are elements that do not arise as a nontrivial product.  Because the underlying multiplicative semigroup of $S_c(ab)$ is $3$-nilpotent (every product of length $3$ is $0$) every nonzero element is either prime or the product of two primes, and the same is true for every semiring in the variety $\mathsf{V}(S_c(ab))$.
\begin{pro}\label{pro: S_c(ab)}
$\mathsf{V}(S_c(ab))$ is determined within the variety $\textbf{F}$ by the identities
\begin{align}
xy        & \approx yx; \label{i1}\\
x^2y      & \approx x^2;    \label{i2}\\
x_1x_2x_3 & \approx y_1y_2y_3.\label{i3}
\end{align}
\end{pro}
\begin{proof}
It is easy to verify that $S_c(ab)$ satisfies the identities (\ref{i1}), (\ref{i2}) and (\ref{i3}) and so does the variety $\mathsf{V}(S_c(ab))$.
In the remainder it is enough to show that every finite nontrivial subdirectly irreducible
member of $\textbf{F}$ that satisfies (\ref{i1}), (\ref{i2}) and (\ref{i3}) is a member of $\mathsf{V}(S_c(ab))$.

Let $S$ be such an algebra. By Lemma \ref{lem: flat} $(i)$ we have that
it is flat and $(S, \cdot)$ is $0$-cancellative and has a unique $0$-minimal ideal $I$.
As the identity (\ref{i3}) holds in~$S$, it follows that $I$ contains exactly one nonzero element.
Assume that $I=\{0, \omega\}$.
If $S$ satisfies the identity $x_1x_2\approx y_1y_2$, then $S=\{0, \omega\}$ and so
it is isomorphic to $S_c(a)$. Otherwise, there exist prime elements $a$ and $b$ of $S$ such that
$ab=\omega$. Also, we have that $S^2=\{0, \omega\}$, i.e., $\omega$ is the only nonzero composite
element of $S$.
Thus the multiplication table of $S$ consists of entries that are either $0$ or $\omega$.
The $0$-cancellative property ensures that no row nor column of this table
contains two appearances of $\omega$.  But as every nonzero element divides $\omega$,
it follows that every row and every column (except for the row and column of $0$ and $\omega$) contain $\omega$
(so therefore exactly one $\omega$).  Since $S$ satisfies (\ref{i1}) and (\ref{i2}),
we have that the multiplication table is symmetric and every element on the diagonal is $0$.

Next, we shall show that $S$ lies in $\mathsf{V}(S_c(ab))$.
Assume that $m$ is a positive integer such that $|S|\leq 2^m+2$.
Let $A$ denote the direct product of $m$ copies of $S_c(ab)$.
Let $T$ denote the subsemiring of $A$ generated by all $m$-tuples
consisting of $a$ or $b$ in each coordinate.
If $I$ denotes the set of all elements of $T$ containing a $0$ coordinate,
then  the Rees congruence $\rho_{_I}$ is a semiring congruence on $T$.
It is easily verified that the quotient algebra $T/ I$ contains a copy of
$S$. This completes the proof.
\end{proof}

Notice that both (\ref{i1}) and (\ref{i2}) hold in $S_c(abc)$. By Lemma \ref{lem: flat} $(i)$ and Proposition~\ref{pro: S_c(ab)} we immediately have the following corollary.
\begin{cor}\label{cor: S_c(ab)}
$\mathsf{V}(S_c(ab))$ is finitely based and is determined within the variety $\mathsf{V}(S_c(abc))$ by the identity $(\ref{i3})$.
\end{cor}

The following result classifies the subvarieties of $\mathsf{V}(S_c(abc))$.
\begin{pro}\label{pro: S_c(abc)}
$\mathsf{V}(S_c(abc))$ has exactly four subvarieties: $\mathsf{V}(S_c(abc))$,
$\mathsf{V}(S_c(ab))$, $\mathsf{V}(S_c(a))$ and the trivial variety.
\end{pro}

\begin{proof}
Suppose that $\mathcal{V}$ is an arbitrary nontrivial subvariety of $\mathsf{V}(S_c(abc))$.
Since the identity $x_1x_2x_3x_4\approx y_1y_2y_3y_4$ holds in $S_c(abc)$, we need only
consider the following cases:
\begin{itemize}
\item $\mathcal{V}\models (\ref{i4})$. By Lemma \ref{lem: flat} $(iii)$ we have that $\mathcal{V}=\mathsf{V}(S_c(a))$.

\item $\mathcal{V}\models (\ref{i3})$, $\mathcal{V}\not\models (\ref{i4})$.
Then by Lemma \ref{lem: flat} $(i)$ there exists a finite subdirectly irreducible flat semiring $S$ in $\mathcal{V}$ failing~$(\ref{i4})$. This implies that there exist distinct elements $a_1$ and $a_2$ in $S$ such that
$a_1a_2\neq 0$. Furthermore, the subsemiring of $S$ generated by $\{a_1, a_2\}$ is $\{0, a_1a_2, a_1, a_2\}$
and so it is isomorphic to $S_c(ab)$. Thus $\mathsf{V}(S_c(ab))$ is a subvariety of $\mathcal{V}$.
Since the identity~$(\ref{i3})$ holds in $\mathcal{V}$, we have by Corollary \ref{cor: S_c(ab)} that
$\mathcal{V}$ is a subvariety of $\mathsf{V}(S_c(ab))$. Hence $\mathcal{V}=\mathsf{V}(S_c(ab))$.

\item $\mathcal{V}\not\models (\ref{i3})$. Then $\mathcal{V}=\mathsf{V}(S_c(abc))$. Its proof is similar to that of the preceding case.
\end{itemize}
This completes the proof.
\end{proof}

As a consequence, we have the following corollary.
\begin{cor}
The lattice of subvarieties of $\mathsf{V}(S_c(abc))$ is a $4$-element chain.
\end{cor}

By Lemma \ref{lem: flat} $(ii)$ and $(iii)$, Corollary \ref{cor: S_c(ab)} and Proposition \ref{pro: S_c(abc)} we can prove the main theorem of this section:
\begin{thm}\label{thm2}
$\mathsf{V}(S_c(abc))$ is a limit variety.
\end{thm}

Notice that $S_c(abc)$ satisfies the identity (\ref{i2}).
It follows that $\mathsf{V}(S_c(abc))$ contains no flat extensions of groups.
On the other hand,
from \cite[Lemma 2.10]{jrz} we know that $S_c(abc)$ lies in the variety of $M(a)$ (also known as $S_7$).
By Theorem \ref{thm2} we immediately deduce that $\mathsf{V}(S_c(abc))$ is a limit subvariety of ${\mathsf{V}}(M(a))$. In fact, one can show that it is the unique limit subvariety of ${\mathsf{V}}(M(a))$.
The details will appear in another paper whose aim is to describe subvarieties of ${\mathsf{V}}(M(a))$.

\section{Other limit varieties}
In this section we prove the existence of other limit varieties in the combinatorial (trivial subgroup) case.
We consider semirings of the form $S({\bf w})$ where ${\bf w}$ is a word.  In the case that ${\bf w}$ is square free, the semiring $S({\bf w})$ satisfies the anticommutative condition
\[xy+yx\approx xy+yx+z\approx (xy+yx)z\approx z(xy+yx).
\]
This law fails in any variety containing $S_c(ab)$, so any nonfinitely based anticommutative variety in which all subgroups are trivial must contain a limit subvariety other than those of Section \ref{sec:groups} and Theorem \ref{thm2}.  We present several ways to find such semirings.

For any word ${\bf v}\in X^+$ we say that a word ${\bf w}\in X^+$ is \emph{${\bf v}$-free} if all of its subwords lie outside of the orbit of ${\bf v}$ under endomorphisms of $X^+$.  If ${\bf w}$ is ${\bf v}$-free, then $S({\bf w})\models {\bf v}\approx {\bf v}+z\approx z{\bf v}\approx {\bf v}z$, which we refer to as the \emph{${\bf v}$-free laws}.  In the flat semiring case, $k$-nilpotence coincides with $x_1\dots x_k$-free law.

Section \ref{sec:groups} shows that there are continuum many limit varieties of ai-semirings, infinitely many of which are finitely generated.  All of these examples have a group-theoretic nature, while Theorem \ref{thm2} gives an explicit example of a combinatorial limit variety.  In this section we demonstrate the existing of at least two further examples of limit varieties, one of which cannot be generated by any flat semiring.
\subsection{Limit varieties via the max-plus algebra}
One of the most important examples of an ai-semiring is the max-plus algebra on the non-negative integers $\omega$, where addition of natural numbers (denoted $+$) distributes over the $\max$ operation (denoted $\vee$).  Aceto, \'Esik and Ing\'olfsd\'ottir \cite{aei} show that the max-plus algebra is not finitely based.  The argument presented in \cite{aei} is in the signature including $0$ as a constant, however the proof holds in the signature $\{+,\vee\}$ also, as we now observe.  For each $n$ they construct an equation $e_n$ and a semiring $S_n$ such that $e_n$ is satisfied by the max-plus algebra and not by $S_n$, and such that all $n$-generated subalgebras of~$S_n$ lie in the variety generated by the max-plus algebra.  These properties continue to hold in the signature $\{+,\vee\}$ because the identities~$e_n$ do not involve the constant~$0$.

\begin{thm}\label{thm:maxplus}
There is a limit variety contained in the variety of the max-plus algebra, and it contains only finitely based flat semirings.
\end{thm}
\begin{proof}
Every nonfinitely based variety contains a limit subvariety.  For the second claim, observe that the max-plus algebra satisfies $(x+y)\vee x\approx (x+y)$.  Translated to the default $+,\cdot$ notation this becomes $xy+x\approx xy\approx yx$.  A flat semiring satisfying this law is finitely based. To see this, consider a subdirectly irreducible flat semiring~$S$ in the variety defined by $xy+x\approx xy\approx yx$, with top element $0$ and let $\omega$ be a nonzero element in the smallest $0$-minimal multiplicative ideal of $S$.  Applying $xy+x\approx xy$ with $x=y=\omega$ we have that $\omega\leq \omega^2$, and as $S$ is flat we have either $\omega^2=0$ or $\omega^2=\omega$, so that $\{0, \omega\}$ is a subuniverse of $S$.  If
$S=\{0, \omega\}$ then it is isomorphic to either $S(a)$ or $M(1)$, according to whether $\omega\omega=0$ or $\omega\omega=\omega$.  So assume that there are $a, b\neq 0$ such that $ab=\omega$.  Then applying $xy+x\approx xy$ and $xy\approx yx$, we obtain $ab+a=ab=ba=ba+b$.  As $S$ is flat and $ab=\omega$, this implies that $a=\omega=b$ so that we have $M(1)$ again. From~\cite{sr} we know that the variety generated by $S(a)$ and $M(1)$ is hereditarily finitely based. This completes the proof.
\end{proof}
%
%
%

\subsection{Limit varieties via Lee words}
There are many semigroup varieties that are known to have subvariety lattices of cardinality $2^{\aleph_0}$, and this has usually been proved by way of word patterns yielding independent systems of equations; see Jackson \cite{jac00}.  As there are only countably many finitely based varieties in any finite signature, this yields nonfinitely based subvarieties (and hence the existence of limit varieties).  We now show that such patterns can be adapted to the semiring setting, yielding anticommutative combinatorial varieties without a finite basis, and hence limit varieties not covered by Section \ref{sec:groups} or Theorem \ref{thm2}.

One such pattern is
\[
\bell_n:=x_1x_2x_1x_3x_2x_4x_3\dots x_{n}x_{n-1}x_n=x_1\cdot \Big(\prod_{1\leq i\leq n-1} (x_{i+1}x_i)\Big)\cdot x_n.
\]
In the literature, this chain pattern has usually included some letters between the first occurrence of $x_1$ and $x_2$, and between the final occurrences of $x_{n-1}$ and $x_n$, but the essential ingredient is the pattern displayed.  The pattern first appeared in the work of Lee and co-authors \cite{Lee14a,Lee14b,LZ14}, and some researchers refer to them as \emph{Lee words}.  They have been used extensively for independent patterns in Jackson~\cite{jac:irr} and Jackson and Lee \cite{jaclee} in the context of monoids, though for our purposes we will not require an identity element and are working in the context of semirings.  The theme of previous use is that for $n\neq m$ the word $\bell_n$ is $\bell_m$-free.  To see this, observe that each  subword of $\bell_n$ of length more than $1$ occurs just once in $\bell_n$, hence any endomorphism that takes $\bell_m$ to a subword of $\bell_n$ must send individual letters to individual letters; but it is clear that this is only possible if $n=m$.  This observation establishes the following lemma.
\begin{lem}\label{lem:leen}
$S(\bell_n)$ satisfies the $\bell_m$-free laws if and only if $n\neq m$.
\end{lem}
\begin{thm}\label{thm:anotherlimit}
The equation system $\Sigma$ that consists of
\begin{itemize}
\item the $\bell_n$-free laws \up(for all $n$\up)\up,
\item the $x^2$-free laws\up,
\item the laws defining anticommutative ai-semirings
\end{itemize}
is not finitely based.
\end{thm}
\begin{proof}
By the compactness of equational logic, it suffices to show that no finite subset of $\Sigma$ is equivalent to the whole system.  For any finite subset $\Sigma'$ of $\Sigma$ there is $n$ such that $\bell_n$-free law is not included in $\Sigma'$.  Then $S(\bell_n)$ satisfies $\Sigma'$ because it is an anticommutative  flat semiring satisfying the $x^2$-free laws and all of the $\bell_m$-free laws in $\Sigma'$ (where $m$ is necessarily not $n$) by Lemma \ref{lem:leen}. However, $S(\bell_n)$ fails the $\bell_n$-free law that is in $\Sigma$, so that $\Sigma'$ is not a basis for the variety defined by~$\Sigma$.
\end{proof}
There are many other such patterns that can be used in a similar way.  The pattern
\[
\bs_n=y_0x_1x_2x_3x_4x_5y_0\Big(\prod_{1\leq i\leq n}(y_ix_{i+5}y_i)\Big)y_{n+1}x_{n+6}x_{n+7}x_{n+8}x_{n+9}x_{n+10}y_{n+1}
\]
for example was introduced by Sapir and Volkov \cite{sapvol} and forms the basis of many examples in \cite{jac00}; a version of Theorem \ref{thm:anotherlimit} holds with minimal change to statement and proof.  Another obvious system of patterns is
\[
\bp_n=x_0^3\big(\prod_{1\leq i\leq n}y_i^2\big)x_1^3.
\]
Again a version of Theorem \ref{thm:anotherlimit} will hold, though this time we cannot assume the square-free law, nor anticommutativity.  We can instead use satisfaction of the law $xyz+yzx\approx xyz+yzx+xxx$; this is easily verified as holding on $S(\bp_n)$ because the only nonzero evaluations of $xyz+yzx$ are $x=y=z=x_0$ or $x=y=z=x_1$.  The law obviously fails on $S_c(abc)$ (for example, when $x=a,y=b,z=c$), so that $S_c(abc)$ is not in the variety generated by $\{S(\bp_n)\mid n\in\mathbb{N}\}$.  Other examples abound.

In the Section \ref{sec:appendix} we provide an example generator for a nonfinitely based subvariety of the system defined in Theorem \ref{thm:anotherlimit}, as well as some other peripheral consequences of the ideas in this section.

\subsection{Limit varieties in the signature $\{+,\cdot,1\}$}
It follows from \cite[Theorem~6.8]{jrz} that in the signature $\{+,\cdot,1\}$ of semirings with multiplicative identity $1$,  the 3-element semiring with identity $M(a)$ generates a limit variety. In the signature $\{+,\cdot\}$, the semiring variety of $M(a)$ properly contains $S_c(abc)$, so $M(a)$ no longer generates a limit variety there.

\section{A nonfinitely based ai-semiring within the variety described by Theorem \ref{thm:anotherlimit}}\label{sec:appendix}
There is value in providing a nonfinitely based semiring satisfying the laws~$\Sigma$ of Theorem \ref{thm:anotherlimit}.  For this purpose it is convenient to introduce two further word constructions.  The first is similar to an infinite limit of the sequence $(\bell_n)_{n\geq 1}$ which we denote by $\bell_\omega$. In order to make the definition precise it is first convenient to think of the variable indices in $\bell_n$ as integers modulo $n+1$.  In this way, the largest index $n$ in $\bell_n$ can also be denoted $-1$, and we may read the word right-to-left as descending through indices $-1,-2,\dots$ until the start of the word is reached at $-n$ (there is no $x_0$).  Now we may similarly consider an infinite string $\bell_\omega$ defined on the nonzero integers:
\[
x_1x_2x_1x_3x_2x_4\dots x_{-4}x_{-2}x_{-3}x_{-1}x_{-2}x_{-1}
\]
In other words, the ordering of the index set is the usual order on positive integers and on negative integers, except that all negative integers are placed above the positive integers.   We extend the notion of subword to this infinite string by allowing all finite subwords as subwords as well as all infinite substrings, provided they have a start point and an end point: thus they must start at an occurrence of some $x_i$ for positive $i$ and finish at an occurrence of $x_j$ for some negative $j$.

The second word construction is a variation of the words $\bell_n$.  Let $1<i<n-1$ and consider the result $\bk_{n,i}$ of replacing the first occurrence of $x_i$ and $x_{i+1}$ by $y_i$ and $y_{i+1}$ respectively, and the second occurrence by $z_i$ and $z_{i+1}$.
\[
\bk_{n,i}:=x_1x_2x_1x_3x_2x_4x_3\dots x_{i-2}y_ix_{i-1}y_{i+1}z_ix_{i+2}z_{i+1}x_{i+3}\dots x_{n}x_{n-1}x_n
\]

By a nonletter word (or subword) we mean a word of length at least $2$.  Let $\bw$ be any word in which all nonletter subwords appear just once; the words $\bell_n$ and $\bk_{n,i}$ are examples.  Each nonletter subword of $\bw$ can be identified with the position of its first and last letter within the word $\bw$.  So for example in the word $\bell_n$, the subword $x_2x_1$ starts at position $2$ and finishes at position $3$ and so may recorded by the pair $(2,3)$.  Multiplication between nonletter subwords of $\bw$ in $S(\bw)$ is then simply $(i_1,j_1)(i_2,j_2)=(i_1,j_2)$ when $i_2=j_1+1$.  Multiplication between individual letters $a,b$ is similar: $a\cdot b=(i,i+1)$ provided that letter occurs in position $i$ and letter $b$ in position $i+1$ within $\bw$.  Similarly, $a\cdot (i,j)=(i-1,j)$ if $a$ appears in position $i-1$ and $(i,j)\cdot a=(i,j+1)$ if $a$ appears in position $j+1$.  All other products are zero.  We refer to this description of multiplication in $S(\bw)$  as the \emph{bracket representation}.  The bracket representation idea extends to our negative index convention, and then also to the infinite word such as $\bell_\omega$, where the indexing can follow the same convention we are using for variable indices: all negative numbers are larger than all positive numbers (and $0$ is omitted).  Thus the bracket $(3,-2)$ in $\bell_\omega$ corresponds to the string $x_1x_3x_2x_4\dots x_{-3}x_{-1}x_{-2}$.

The bracket representation is useful because it makes clear that the only difference between $S(\bell_n)$ and $S(\bk_{n,i})$ are at the level of products involving letter elements.  This plays a role in the next lemma.

\begin{lem}\label{lem:matching}
If $n\leq m\leq \omega$ the word $\bell_m$ contains an image $\varphi$ of $\bk_{n,i}$ by matching the first $i-1$ generators $x_1,\dots,x_{i-1}$ to themselves, the final $n-i-1$ variables $x_{1+i-n},\dots, x_{-1}$ to the corresponding final $n-i-1$ variables with the same negative index, matching $y_i$ with $x_i$, $z_{i+1}$ with $x_{i-n}$, and any splitting of the central portion of $\bell_m$ into a nonempty prefix and suffix: $\varphi(y_{i+1}z_i)=x_{i+1}x_i\cdot x_{i+2}x_{i+1}\cdots x_{2+i-n}x_{1+i-n}$.
\end{lem}
\begin{proof}
This is a simple matching of patterns.  In the following schematic we display the symbol~$\cdot$ between  blocks to aid clarity of the pattern and let~$p$ denote the negative number $1+i-n$.
\begin{center}
\begin{tikzpicture}
\node at (-8.5,1) {$\bk_{n,i}$:};
\node at (-2.125,1) {$ \cdots x_{i-1}x_{i-2}\cdot y_ix_{i-1}\, \quad\quad {\cdot} \quad\quad\,  y_{i+1}z_i\quad\quad\ {\cdot} \quad\quad\, x_{i+2}z_{i+1}\ \cdot x_{i+3}x_{i+2} \cdots $};
\node at (-8.5,-0.2) {$\bell_m$:};
\node at (-2.05,-0.2) {$ \cdots x_{i-1}x_{i-2}\cdot x_ix_{i-1}\, \cdot \, x_{i+1}x_i\cdot x_{i+2}x_{i+1}\cdots  x_{p}\, \cdot \, x_{p+2}x_{p+1}\cdot x_{p+3}x_{p+2} \cdots $};
\draw (-6.95,0.9) [->] to  (-6.95,.05);
\draw (-6.25,0.9) [->] to  (-6.25,.05);
\draw (-5.35,0.85) [->] to  (-5.35,.05);
\draw (-5.05,0.9) [->] to  (-5.05,.05);
\draw (-2.4,0.65) [->] to  (-2.4,.15);
\draw [decoration={brace}, decorate,thick] (-4.1,0) -- (-0.7,0);
\draw [decoration={brace}, decorate,thick] (-1.91,0.8) -- (-2.885,0.8);

\draw (-.15,0.9) [->] to  (-0.15,.05);
\draw (.55,0.9) [->] to  (0.55,.05);
\draw (1.52,0.9) [->] to  (1.52,.05);
\draw (2.25,0.9) [->] to  (2.25,.05);
\end{tikzpicture}
\end{center}
Note that in the case of $n=m$ then $p\equiv i\mod n+1$ so that $\phi(y_{i+1}z_i)=x_{i+1}x_i$.  In this case there is only one possible split into prefix and suffix, but for $m>n$ there is a choice.  For $m=\omega$ either the prefix or the suffix \up(but not both\up) must be infinite due to our restriction that all infinite subwords of $\bell_\omega$ have a start letter and a finish letter.
%
%
%
\end{proof}
A consequence of Lemma \ref{lem:matching} is the following lemma.
\begin{lem}\label{lem:mversusn}
$S(\bk_{n,i})$ is isomorphic to a subsemiring of $S(\bell_m)$ whenever $m\geq n+2$.
\end{lem}
\begin{proof}
We use $\varphi$ to denote the matching described in Lemma \ref{lem:matching}, and as $m\geq n+2$ we may split the word $x_{i+1}x_i\cdot x_{i+2}x_{i+1}\cdots x_{2+i-n}x_{1+i-n}$ as
\[
\phi(y_{i+1})=x_{i+1}x_{i}\cdot x_{i+2}\text{ and }\phi(z_i)=x_{i+1}\cdot x_{3+1}x_{i+2}\dots x_{2+i-n}x_{1+i-n}.
\]
Now, distinct subwords of $\bk_{n,i}$ correspond to distinct subwords of $\bell_m$ under $\varphi$.  Moreover it is easy to see by inspection that the nonzero products in $S(\bk_{n,i})$ are the only way to form nonzero products in $S(\bell_m)$ using the elements of the form $\varphi(a)$ for variable $a$ (that is, the subalgebra of $S(\bell_m)$ generated by the image under $\varphi$ of the generators of $S(\bk_{n,i})$).  To see this it is easy to use the bracket representation or to examine the displayed matching in the proof of Lemma \ref{lem:matching}.  In the displayed matching it can be seen that for any such generators $a,b$, the nonzero products $\varphi(a)\varphi(b)$ in $S(\bell_m)$ are when $ab$ is nonzero in $S(\bk_{n,i})$.  To argue via the bracket convention, notice that we have chosen $\varphi(y_{i+1})$ to finish at a position that does not match any an occurrence of any $\varphi(x_j)$ nor of $\varphi(y_i)$ and $\varphi(z_{i+1})$, and dually for the starting position of $\varphi(z_i)$.
\end{proof}
Lemma \ref{lem:mversusn} can be adapted to accommodate the case $m=n+1$ if we replace ``subsemiring'' with ``quotient of a subsemiring'', but as we do not need it here we do not give further details.

For $1<i<n-1$, the elements $x_i$ and $x_{i+1}$ are prime in $S(\bell_n)$ and $S(\bell_n)\backslash\{x_i,x_{i+1}\}$ is a subsemiring, which we denote by $T_i(\bell_n)$.   Similarly, we may let $T_i(\bk_{n,i})$ denote the subsemiring of $S(\bk_{n,i})$ on $S(\bk_{n,i})\backslash \{y_i,y_{i+1},z_i,z_{i+1}\}$: generated by all elements that are not individual letters of index $i$ or $i+1$.
\begin{lem}\label{lem:isom}
For all $1<i<n-1$ we have
$T_i(\bell_n)\cong T_i(\bk_{n,i})$.
\end{lem}
\begin{proof}
This is a trivial consequence of the bracket representation of $S(\bell_n)$ and $S(\bk_{n,i})$: we have removed the only generators that distinguished the two semirings: all remaining letter elements sit in the same positions in the words $\bell_n$ and $\bk_{n,i}$, and the bracket elements are identical in both.
\end{proof}
In the following lemma we allow $m=\omega$, but $n$ is finite.
\begin{lem}\label{lem:Leesubword}
If $n>2k+4$ and $m\geq n+2$ then all $k$-generated subsemirings of $S(\bell_n)$ lie in $\mathsf{V}(S(\bell_m))$.
\end{lem}
\begin{proof}
As $n>2k+4$, every $k$-generated subsemiring of $S(\bell_n)$ must omit two consecutive generators from the list  $x_3,\dots,x_{n-2}$ and hence must be a subsemiring of $T_i(\bell_n)$ for some $i$.  Now, $T_i(\bell_n)$ is isomorphic to $T_i(\bk_{n,i})$ by Lemma \ref{lem:isom} and $T_i(\bk_{n,i})$ is a subsemiring of $S(\bk_{n,i})$, itself a subsemiring of $S(\bell_m)$ by Lemma \ref{lem:mversusn}.
\end{proof}

\begin{thm}\label{thm:Lee}
The semiring $S(\bell_\omega)$ is not finitely based as a semiring.  So too is the semiring $S(\{\bell_n\mid n\in M\})$, for any infinite subset $M\subsetneq \mathbb{N}$ with infinite complement.
\end{thm}
\begin{proof}
No $S(\bell_n)$ is in the variety of $S(\bell_\omega)$ because $S(\bell_\omega)$ satisfies the $\bell_n$-free laws for every $n$.  Similarly, the variety of $S(\{\bell_n\mid n\in M\})$ omits $S(\bell_n)$ for every $n\notin M$.
Now let $k$ be an arbitrary positive integer.  We show that the $k$-variable laws of $S(\bell_\omega)$ and of $S(\{\bell_n\mid n\in M\})$ are not a basis for the full equational theory, from which it trivially follows that no finite basis is possible.  In the case of $S(\bell_\omega)$, this is because $S(\bell_n)$ satisfies all $k$-variable laws of $S(\bell_\omega)$ provided $n>2k+4$, by Lemma~\ref{lem:Leesubword}.   In the case of $S(\{\bell_n\mid n\in M\})$ choose any $n>2k+4$ with $n\notin M$.  Then $S(\bell_n)$ is not in the variety of $S(\{\bell_n\mid n\in M\})$ but we may choose $m\in M$ with $m\geq n+2$ so that all $k$-generated subsemirings of $S(\bell_n)$ lie in $S(\bell_m)$, a subvariety of $S(\{\bell_n\mid n\in M\})$.
\end{proof}

Semirings of the form $S(\bw)$ for words $\bw$ (and for sets of words) were suggested for exploration in \cite{jrz}, and Theorem \ref{thm:Lee} appears to provide further impetus for the potential complexity of variety type problems that might arise there.  We finish the section with some basic lemmas that might help guide future use of the construction.  The free ai-semiring $F_X$ freely generated by $X$ is isomorphic to the semiring of finite nonempty subsets of the free semigroup $X^+$, and all relatively free ai-semirings are quotients of this by way of a full invariant congruence.  Following a similar notion for semigroups, we say that a semigroup word $\bw$ is a \emph{semiring isoterm} for an ai-semiring $S$, if in the fully invariant congruence $\theta$ on $F_X$ for the variety of $S$, the element $\{\bw\}$  is in a singleton $\theta$-class.  The notion of semiring isoterm turns out to correspond to the notion of a \emph{minimal word} as defined by Dolinka in \cite{dol09}.  A word~$\bw$ is a minimal word for $S$ if $S\models \bv \lesssim\bw$ implies $\bw=\bv$ (where $s\lesssim t$ abbreviates the identity $s+t\approx t$).  To see this coincides with the semiring isoterm concept (for additively idempotent semirings), observe first that if $S\models \bv\lesssim \bw$ and $\bv\neq \bw$ then~$\{\bw\}$ is congruent to $\{\bw,\bv\}$ with respect to $\theta$.  Conversely, if $\{\bw\}$ is congruent to some finite set of words $A$, then any individual word $\bv\in A$ has $S\models \bv\lesssim \bw$. The article \cite{dol09} makes use of a stronger notion called \emph{isolated word}, but it is the notion of isoterm described here that enables the following direct analog of well known facts for semigroups.
\begin{lem}\label{lem:isoterm}
Let $\bw\in X^+$ be a semigroup word.  Then $\bw$ is an isoterm for an ai-semiring $S$ if and only if the flat semiring $S(\bw)$ is in the variety of $S$.
\end{lem}
\begin{proof}
Let $F$ be the relatively free semiring for the variety of $S$ in a suitable generating set for $\bw$ and let $\theta$ be the corresponding fully invariant congruence giving rise to $F$ from the free ai-semiring $F_X$.  If $\bw$ is a semiring isoterm then so too are all subwords of $\bw$, because if $\bw=\bw_1\bw_2\bw_3$ (where $\bw_1$ and $\bw_2$ are possibly empty subwords) and $\bv\lesssim \bw_2$ holds then $\bw_1\bv\bw_3\lesssim \bw$ holds also.  So every subword $\bw'$ of~$\bw$ is in a singleton class of $\theta$ and we may factor  $F$ by the order ideal consisting of all elements except those of the form $\{\bw'\}$, where $\bw'$ is a subword of $\bw$.  This is isomorphic to $S(\bw)$.  For the other direction it is trivially true that $\bw$ is an isoterm for $S(\bw)$, from which it follows that $\bw$ is an isoterm for any variety containing~$S(\bw)$.
\end{proof}
Examples in the style of Jackson \cite{jac00} are easily translated using Lemma \ref{lem:isoterm}.  It will follow easily for example, that any ai-semiring variety in which all of the Sapir-Volkov words $\bs_n$ are semiring isoterms has a continuum of subvarieties.
\begin{example}
The variety generated by the semiring $M(abacdc)$ contains $S(\bs_n)$ for all $n$ and hence has a continuum of subvarieties.
\end{example}
\begin{proof}
For any semigroup word $\bv$ and list of variables $L$ we let $\bv(L)$ denote the result of deleting all letters except those in $L$.  Any evaluation of the letters in $\bv$ into the elements of $M(abacdc)$ determines an evaluation of $\bv(L)$ into $M(abacdc)$, where $L$ is the list of letters that are not assigned the value $1$.

For every pair of distinct letters $x,y$ in $\bs_n$ there is a list $L$ containing $x,y$ such that $\bs_n(L)$ is equal to a subword of $abacdc$ up to change of letter names:
\begin{itemize}
\item for $x,y\in \{x_1,\dots,x_{n+10}\}$ we have that $\bs_n(x,y)$ is equivalent up to letter names to the subword $ab$.
\item for $x,y\in\{y_0,\dots,y_{n+1}\}$ we may find $i,j\in\{1,\dots,n+10\}$ such that $x_i$ and $x_j$ occur between the two occurrences of $x,y$ and find that $\bs_n(x,y,x_i,x_j)$ is equivalent to $abacdc$ up to letter names.
\item for $x\in \{x_1,\dots,x_{n+10}\}$ and $y\in \{y_1,\dots,y_{n+1}\}$, if $x$ occurs between $y$ then $\bs_n(x,y)=yxy$ which is equivalent to $aba$; otherwise we may choose $x'\in \{x_1,\dots,x_{n+10}\}$ occurring between the two occurrences of $y$ and choose $y'\in  \{y_1,\dots,y_{n+1}\}$ occurring either side of the single occurrence of $x$ to obtain $\bs_n(x,y,x',y')\in\{y'xy'yx'y,yx'yy'xy'\}$, both of which are equivalent to $abacdc$ up to letter names.
\end{itemize}
If $\bw$ is a word for which $M(abacdc)\models \bw\lesssim \bs_n$ then $\bw$ cannot contain letters not occurring in $\bs_n$ because if $z$ appears in $\bw$ but not $\bs_n$ we may assign all letters in $\bs_n$ the value $1$ and $z$ the value $0$ to contradict $M(abacdc)\models \bw\lesssim \bs_n$.  So the alphabet of $\bw$ is contained in that of $\bs_n$.  The three dot points now ensure that $\bw(L)$ must coincide with $\bs_n(L)$ in each of the choices of $L$ in the three dots points.  This ensures that all of the letters of $\bs_n$ appear in $\bw$ and in the same pattern of occurrences.  From this, and the fact that $\bw$ contains no further letters, we can deduce that $\bw=\bs_n$, as required.
\end{proof}
The multiplicative reduct of $M(abacdc)$ generates a limit variety of monoids if the identity element is given the status of a nullary \cite[Proposition~5.1]{jac05}.  This result does not translate to semirings, because the semiring $M(a)$ (also known as $S_7$) is a subsemiring $M(abacdc)$ (even if $1$ is included as a nullary) and is nonfinitely based \cite[Corollary 4.11]{jrz}.  As a semiring (in signature $\{+,\cdot\}$), the variety generated by $M(a)$ also contains the limit variety $\mathsf{V}(S_c(abc))$ \cite[Lemma 2.10]{jrz}.  This implies that $\mathsf{V}(M(abacdc))$ contains more than one limit subvariety: that generated by $S_c(abc)$, and the limit variety sitting beneath the variety generated by $\{S(\bs_n)\mid n\in \mathbb{N}\}$.     It remains plausible that $M(a)$ (also known as $S_7$) generates a semiring variety with continuum many subvarieties.

\section{Conclusion}
We have found one infinite family of explicitly described limit varieties and one ad hoc example, as well as demonstrated the existence of one further family (continuum in cardinality) and of two further examples.

Understanding further limit varieties remains of interest.  In this regard it would be particularly interesting to explore subvarieties of the variety generated by the max-plus algebra, toward possible explicit identification of a limit subvariety. A related problem is to precisely identify a finitely generated limit variety of ai-semirings that is not generated by a flat algebra (cf.~Theorem \ref{thm:maxplus}).

The dual approach to identifying limit varieties is to broadly describe some general laws that guarantee finite axiomatisability for varieties.  While there are a number of initial efforts in this direction, there is no candidate analogue of even the earliest results of this kind due to Perkins for semigroups \cite{per}: every  variety of commutative semigroups and every  periodic variety of permutative semigroups have a finite identity basis.  This perhaps reflects the possibility that many more ai-semirings are without a finite basis than the corresponding situation for semigroups, though it is premature to make a clear conjecture on this, nor is it clear how to formulate such an idea precisely.

\bibliographystyle{amsplain}


\end{document}